\documentclass{amsart}%

\usepackage{hyperref}%
\usepackage{amsmath, amsthm, amssymb, amsfonts}
\usepackage [foot]{amsaddr}%
\usepackage[url=false,backend=bibtex,isbn=false,firstinits]{biblatex}%
\newcommand{\EE}{{\mathbf E}}
\newcommand{\dd}{ { \mathrm{d}} }

\newcommand{\cF}{{ \mathcal F}}
\newcommand{\cD}{{ \mathcal D}}

\newcommand{\HS}{\mathrm{HS}}

\newcommand{\ts}{k}
\newcommand{\dt}{\ts}

\newcommand{\Wj}{\Delta W^j}

\newcommand{\dH}{\dot{H}}
\newcommand{\wH}{\dH}
\newcommand{\Hoz}{H_0^1}

\newcommand{\N}{\mathbb{N}}

\DeclareMathOperator*{\Dom}{\mathrm{Dom}}
\newcommand{\dom}{\Dom}
\newtheorem{thm}{Theorem} [section]
\newtheorem{cor} [thm] {Corollary}
\newtheorem{lem} [thm] {Lemma}

\makeatletter
\renewcommand{\email}[2][]{%
  \ifx\emails\@empty\relax\else{\g@addto@macro\emails{,\space}}\fi%
  \@ifnotempty{#1}{\g@addto@macro\emails{\textrm{(#1)}\space}}%
  \g@addto@macro\emails{#2}%
}
\makeatother

\begin{document}

\title[Pathwise existence]{Pathwise existence of solutions to the Implicit Euler method
  for the stochastic Cahn-Hilliard Equation}

\author{Daisuke Furihata}
\email{furihata@cmc.osaka-u.ac.jp}
\author{Fredrik Lindgren}%
\address[Furihata and Lindgren]{{Cybermedia Center, Osaka University} {1-32 Machikaneyama,
Toyonaka, Osaka 560-0043, Japan}}
\email{fredrik.lindgren1979@gmail.com}

\author{Shuji Yoshikawa}
\address[Yoshikawa]{{Graduate School of Science and Engineering, Ehime
  University} {3 Bunkyo-cho, Matsuyama, Ehime 790-8577, Japan}}
\email{yoshikawa@ehime-u.ac.jp}

\begin{abstract}
We consider the implicit Euler approximation of the stochastic Cahn-Hilliard equation driven by
  additive Gaussian noise in a spatial domain with smooth boundary in
  dimension $d\le 3$. %
We show pathwise existence and uniqueness of
 solutions for the method under a restriction on the step size that is
 independent of the size of the initial value and of the increments of
 the Wiener process.  This result also relaxes the imposed assumption on the
 time step for the deterministic Cahn-Hilliard equation assumed in
 earlier existence proofs.
\end{abstract}

\maketitle

\section{Introduction}
Let $\mathcal{D}\subset \mathbb{R}^d$, $d\le 3$, be a bounded spatial domain with smooth boundary $\partial\cD$
and consider the stochastic Cahn-Hilliard equation written in abstract  form,
\begin{equation}\label{eq:chc}
\begin{aligned}
\dd X+A(A X+f(X))\,\dd t&=\dd W(t),~t\in(0,T],\quad X(0)=X_0,
\end{aligned}
\end{equation}
where $A$ is the realisation of the Laplace operator $-\Delta$ with
homogenous Dirichlet boundary conditions in $H=L^2(\cD)$  with inner
product $\langle\cdot\,,\cdot\rangle$ and induced norm
$\|\cdot\|$. The non-linearity $f$ is given by $f(s)=s^3-\beta^2 s$
and $W$ is an $H$-valued $Q$-Wiener process.

Existence of solutions to \eqref{eq:chc} is studied in
\cite{Prato} and spatial semi-discretisation with a finite element
method in \cite{KLM} and \cite{KLMerr}. Here, we are interested in existence and
uniqueness of the implicit Euler approximation of \eqref{eq:chc} given
by
\begin{equation}\label{eq:be}
\begin{aligned}
X^j+\dt A^2X^j+\dt Af(X^j)&=X^{j-1}+\Wj,\, j\in I_N,\quad X^0=X_0,
\end{aligned}
\end{equation}
where $I_N=\{1,\dots,N\}$, $N\in \N$, $\dt =T/N$ and $\Wj=W(t_j)-W(t_{j-1})$ for
$t_j=j\dt$, $j\in I_N\cup \{0\}$. That is, we study a temporal semi-discretisation. 

In the deterministic case, when $W=0$, existence is usually proved \cite{Stig,yoshikawa} by
the reformulation of \eqref{eq:be} as a fixed point problem in a ball
$\{\|A^{1/2}x\|\leq M\}$. If $\dt\leq
\dt_0$ the constructed mapping in the formulation
becomes a contraction and existence and uniqueness
follows. However, the constant $M$ grows and $\dt_0$ shrinks as
$\|A^{1/2}X_0\|$ grows.

In the present setting this dependence can not be allowed. At every time step, the
right hand side of \eqref{eq:be} plays the role of the initial value
and, being a Gaussian random variable, $\Wj$ may be arbitrary large
with positive probability. If we would rely on earlier existence
results we would be forced to utilise an adaptive time stepping scheme
and facing the risk of needing arbitrary small time steps. Instead, we
shall prove that the equation 
\begin{equation}\label{eq:ell}
u+\dt A^2u+\dt Af(u)=y
\end{equation}
has a solution in $H^2\cap\Hoz$ as soon as $y\in \dom(A^{-1})$, the
domain of $A^{-1}$. At each time
step, $u$ corresponds to $X^{j}$ and $y$ to $X^{j-1}+\Wj$ in
\eqref{eq:be}, so for this
assumption to hold it is sufficient  that $X_0, \Wj\in \dom(A^{-1})$,
$j\in I_N$, a.s. This holds if, e.g., $\EE \|A^{-1}X_0\|^2<\infty$ and
$\|A^{-1}Q^{1/2}\|_{\HS}<\infty$, where $\|\cdot\|_{\HS}$ denotes
the Hilbert-Schmidt norm in $H$ and $Q$ is the covariance operator of
$W$. More precisely, our main results
are the following.

\begin{thm}\label{thm:ell}
Assume $y\in \dom(A^{-1})$ and $\dt < 4/\beta^4$, then
\eqref{eq:ell} has a unique
solution $x\in H^2\cap\Hoz$.
\end{thm}
\begin{cor}\label{cor:be}
If, a.s., $X_0,
\,\Wj\in\dom(A^{-1})$, $j\in I_N$ and $\dt<  4/\beta^4$, then there is an a.s. unique solution to
\eqref{eq:be} with $X^j\in H^2\cap\Hoz$ for $j\in I_N$.
\end{cor}

We shall prove the existence part of Theorem \ref{thm:ell} by applying
Schaefer's fixed point theorem to the mapping $z=T_y(x)$ given by
\begin{equation}\label{eq:fp1}
z+\dt A^2z+\dt Azx^2=y+\dt\beta^2Ax.
\end{equation} 
Clearly, a fixed point, $z=x$, of \eqref{eq:fp1} is a
solution of \eqref{eq:ell}.

The outline is as follows. In Section \ref{sec:prel} we give some
necessary definitions and state
Schaefer's fixed point theorem and other required results. In
Section \ref{sec:proof} we give the mapping $T_y$  a rigorous meaning
and show that it fulfils the assumptions of Schaefer's fixed point
theorem.

\section{Preliminaries}\label{sec:prel}

We shall use the abbreviation $L_p=L_p(\cD)$ for the standard
function spaces on $\cD$ and $H^r=H^r(\cD)$ refers to the usual
Sobolev spaces with all partial derivatives of order $\leq r$ being
square integrable. The space $\Hoz$ is the completion of
$C^\infty_b(\cD)$ in $H^1$. It hold that the operator $A$ with $\Dom(A)=\Hoz\cap H^2$ has strictly positive eigenvalues
$\lambda_1<\lambda_2\leq\ldots$ diverging to infinity so any real power $A^{s}$ may be defined and $A^{s}$ is
positive definite and self-adjoint with $\dom(A^{s/2})=:\dH^{s}$.   If  $s_1\geq s_2$ then
$\dH^{s_1}\subset\dH^{s_2}$ and 
\begin{equation}\label{eq:Ass}
\|A^{s_2/2}x\|\leq
\lambda_1^{(s_1-s_2)/2}\|A^{s_1/2}x\|.
\end{equation}
 In particular,
$\Hoz=\dH^{1}$ and $H^{-1}:=(\Hoz)^*=\dH^{-1}$. The space
$\Hoz$ is a Hilbert space with the inner product
$\langle\cdot,\cdot\rangle_1:=\langle
A^{1/2}\cdot,A^{1/2}\cdot\rangle$. %
More generally, we have the
family of inner products $\langle \cdot,\cdot \rangle_{s}:=\langle
A^{s/2}\cdot, A^{s/2}\cdot \rangle$ and induced norms
$\|\cdot\|_{s}=\langle \cdot,\cdot\rangle_s^{1/2}$ on
$\dot{H}^s$. We shall use
$\langle\cdot\,,\cdot\rangle$ also for the duality pairing
of $\dH^{s}$ and $\dH^{-s}$.

We will frequently utilise the embeddings $\Hoz\subset L_6\subset L_3$ with
\begin{equation}\label{eq:sob}
c\|u\|_{L_3}\leq \|u\|_{L_6}\leq C\|u\|_1,
\end{equation}
and the resulting inequality
\begin{equation}\label{eq:sobres}
\|u\|_{-1}\leq C\|u\|_{L_{6/5}}.
\end{equation}
The first inequality in \eqref{eq:sob}
holds in arbitrary spatial dimension $d$ while the latter and
\eqref{eq:sobres} hold for $d\leq 3$. See \cite[Lemma 2.5]{KLM} for a proof of
a finite dimensional version, the proof in our case is almost identical. We also have
\begin{align}
|\langle u^2v,w\rangle|&\leq \|u\|_{L_3}^2\|v\|_{L_6}\|w\|_{L_6},\label{eq:Hold1}\\
\|u^2v\|&\leq \|u\|_{L_6}^2\|v\|_{L_6}\text{ and}\label{eq:Hold2}\\
\|uvw\|_{-1}&\leq C\|u\|_{L_6}\|v\|_{L_3}\|w\|_{L_3}, \quad (d\leq 3)\label{eq:Hold3}
\end{align}
where \eqref{eq:Hold1} and \eqref{eq:Hold2} holds for arbitrary $d$ as being consequences of
H\"older's inequality. The third, \eqref{eq:Hold3}, is a consequence
  of \eqref{eq:sobres} and H\"older's inequality.%

The following theorem can be found in
\cite[Theorem 4, Section 9.2]{evans}.
\begin{thm}[Schaefer's fixed point theorem]
Assume that $X$ is a real Banach space and that $T\colon X\rightarrow
X$ is a continuous, compact mapping. If the set  $\cF=\cup_{\lambda\in [0,1]}\{u\in
\Hoz:u=\lambda T u\}$ is bounded, then $T$ has a fixed point.
\end{thm}

The results does not rely on any probabilistic arguments in addition
to the ones used in the introduction. We refer the 
reader to \cite{DPZ}  (the Hilbert-Schmidt norm is defined in Appendix C)
 and \cite{PR}.

\section{Proof of the main theorem}\label{sec:proof}
To make sure that $T_y\colon \Hoz\rightarrow \Hoz$ is well-defined for
every $y\in \dH^{-2}$, in fact even for every $y\in \dH^{-3}$, we let $z=T_y(x)$ be such that for every
$v\in \Hoz$, 
\begin{equation}\label{eq:inner}
\langle z, v\rangle_{-1}+\dt\big( \langle z,v\rangle_1+\langle xz ,x
  v\rangle\big)=\langle \dt\beta^2 x + A^{-1} y, v\rangle.
\end{equation}
This is of the form
$B_x(z,v)=L_{y,x}(v)$ where $B_x$ is an inner product
and $L_{y,x}$ is a bounded linear functional on $\Hoz$ if
$x\in L_3$ and $y\in\dH^{-3}$. That $B_x$ has the claimed domain follows from \eqref{eq:Hold1} %
and that $\Hoz\subset
  L_6\subset L_3$. From \eqref{eq:Hold1} and \eqref{eq:sob},
  we get that $\dt\|u\|^2_1\leq
B_x(u,u)\leq\left(\lambda_1^{-2}+\dt+C \|x\|_{L_3}^2\right)\|u\|_1^2$.
 Thus,
$(\Hoz,B_x)$ is a Hilbert space. Lemma \ref{lem:def} is then immediate from
Riesz representation theorem.
\begin{lem}\label{lem:def}
If $x\in L_3$ and $y\in \dH^{-3/2}$ then \eqref{eq:inner} has a
unique solution $z\in \Hoz$. In particular, $T_y$  is well defined
as a mapping on $\Hoz$.
\end{lem}
We now let $z$ be this solution and consider the system of equations
\begin{align}
\dt Aw&=-z+y\label{eq:sys1}\\
Au&=w-x^2z+\beta^2x.\label{eq:sys2}
\end{align}
By standard elliptic theory, \eqref{eq:sys1} has a unique weak
solution $w\in H$ as soon as $y\in \dH^{-2}$. We then get
a unique weak solution $u\in\Hoz$ to \eqref{eq:sys2} if also $x\in
L_3$. We leave to the reader to check that $u=z$. From \eqref{eq:sys2} and \cite [Theorem 4, Section 6.3]{evans}
\begin{equation}\label{eq:H2}
\|z\|_{H^2}\leq C\|w-x^2z+\beta^2x\|\leq C\left(\|w\|+\|x^2z\|+\|x\|\right).
\end{equation}     
Taking $v=z$ in \eqref{eq:inner} , using the positivity of the third
term in the left hand side, the self-adjointness of $A^{1/2}$ and 
H\"older's  and Cauchy's inequalities
we compute
\begin{equation*}
\begin{aligned}
&\|z\|_{-1}^2+\dt \|z\|_1^2+\dt \|zx\|^2=\dt
\beta^2\langle x,z\rangle+\langle A^{-1}y,z\rangle\\
&\,=\dt
\beta^2\langle A^{-1/2} x,A^{1/2}z\rangle+\langle
A^{-3/2}y,A^{1/2}z\rangle \leq \epsilon \dt\|z\|_1^2 +C\left(\dt \|x\|_{-1}^2+\dt^{-1}\|y\|^2_{-3} \right)
\end{aligned}
\end{equation*}
whith $C=C_\epsilon$. Clearly, there is an $\epsilon >0$ such that 
\begin{equation}\label{eq:zb}
\|z\|^2_{-1}+\dt\|z\|_1^2\leq C\left(\dt\|x\|_{-1}^2+\dt^{-1}\|A^{-3/2}y\|^2 \right).
\end{equation}
It follows from \eqref{eq:sys1}, the properties of $A$ and
\eqref{eq:zb} that
\begin{equation}\label{eq:wb}
\begin{aligned}
\dt\|w\|&\leq \|A^{-1}z\|+\|A^{-1}y\|\leq
\lambda_1^{-1/2}\|z\|_{-1}+\|A^{-1}y\|\\
&\leq C\left( \dt^{1/2}\|x\|_{-1}+\dt^{-1/2}\|A^{-3/2}y\| +\|A^{-1}y\|\right).
\end{aligned}
\end{equation}
From \eqref{eq:Hold2}, \eqref{eq:sob} and \eqref{eq:zb} we also get
\begin{equation}\label{eq:nonlint}
\|x^2z\|\leq \|x\|^2_{L_6}\|z\|_{L_6}\leq
C\|x\|^2_{1}\|z\|_{1} \leq C\|x\|^2_{1}\left(\|x\|_{-1}+\dt^{-1}\|A^{-3/2}y\|\right).
\end{equation}
Insert \eqref{eq:wb} and \eqref{eq:nonlint} into \eqref{eq:H2},
use \eqref{eq:Ass} and Young's inequality, to find that
\begin{equation}\label{eq:H2bnd}
\begin{aligned}
\|z\|_{H^2}&\leq C\Big(\dt^{-1/2}\|x\|_{-1}+\dt^{-3/2}\|A^{-3/2}y\| +\dt^{-1}\|A^{-1}y\| +\\
&+ \|x\|^2_{1}\big(\|x\|_{-1}+\dt^{-1}\|A^{-3/2}y\|\big) +\|x\|\Big)
\leq C_k\big(\|x\|^3_{1}+\|A^{-1}y\|^3\big).
\end{aligned}
\end{equation}
Compactness of $T_y$ then follows from Kondrachov-Rellich's
compactness theorem \cite[Theorem 1, Section 5.7]{evans}. %
   We have the following lemma.
\begin{lem}\label{lem:comp}
If $y\in \dH^{-2}$, then $T_y$ is a compact mapping from $\Hoz$
to $\Hoz$.
\end{lem}

We now want to verify  that  $T_y$ is continuous.
\begin{lem}\label{lem:cont}
The mapping $T_y$ is continuous on $\Hoz$ if $y\in \dH^{-3}$.
\end{lem}
\begin{proof}
Take  $x_1$ and $x_2$ in $\wH^1$
and let $z_1=T_y(x_1)$ and $z_2=T_y(x_2)$. Consider these equations of
the form \eqref{eq:inner} and subtract the latter
 from the former, using $v=z_1-z_2$. We then arrive at
\begin{equation}\label{eq:conteq}  
\begin{aligned}
&\|z_1-z_2\|^2_{-1}+k\|z_1-z_2\|_1^2+k\langle
z_1x_1^2-z_2x_2^2,z_1-z_2\rangle\\
&\quad=k\beta^2\langle x_1-x_2,z_1-z_2
\rangle\leq C\|x_1-x_2\|^2_{1}+\frac12\|z_1-z_2\|^2_{-1}
\end{aligned}
\end{equation}
after also invoking H\"older's and Cauchy's inequalities. Note that 
$z_1x_1^2-z_2x_2^2 =x_1^2(z_1-z_2)+(x_1-x_2)z_2(x_1+x_2)$
and thus
\begin{equation}\label{eq:nonlin}
\langle z_1x_1^2-z_2x_2^2,z_1-z_2\rangle=\|x_1(z_1-z_2)\|^2+\langle (x_1-x_2)z_2(x_1+x_2),z_1-z_2\rangle.
\end{equation}
Further, using H\"older's and Cauchy's inequalities and
\eqref{eq:Hold3}  we get
\begin{equation}\label{eq:nonlinbnd1}
\begin{aligned}
&|\langle (x_1-x_2)z_2(x_1+x_2),z_1-z_2\rangle| =|\langle A^{-1/2}
(x_1-x_2)z_2(x_1+x_2),A^{1/2}(z_1-z_2)\rangle|\\
&\quad\leq \frac12\|
(x_1-x_2)z_2(x_1+x_2)\|_{-1}^2+\frac12\|z_1-z_2\|_1^2\\
&\quad \leq
  C\|x_1-x_2\|_{L_{6}}(\|z_2\|^2_{L_6}+\|x_1\|^2_{L_6}+\|x_2\|^2_{L_6})+\frac12\|z_1-z_2\|_1^2.
\end{aligned}
\end{equation}
Inserting \eqref{eq:nonlin} into \eqref{eq:conteq}, rearranging and applying
\eqref{eq:nonlinbnd1} we find that
\begin{equation}\label{eq:finalbnd}
\begin{aligned}
&\frac12\|z_1-z_2\|_{-1}^2+k\|z_1-z_2\|_1^2+\dt\|x_1(z_1-z_2)\|^2 \\
&\, \leq C\big(\dt\|x_1-x_2\|_{L_{6}}(\|z_2\|^2_{L_6}+\|x_1\|^2_{L_6}+\|x_2\|^2_{L_6})+\|x_1-x_2\|^2_{1}\big)+\frac\dt2\|z_1-z_2\|_1^2.
\end{aligned}
\end{equation}
Subtracting $\frac{\dt}2\|z_1-z_2\|_1^2$ from both sides and multiplying
by 2 we conclude that %
\begin{equation*}
\begin{aligned}
k\|z_1-z_2\|_1^2\leq
C\Big(\|x_1-x_2\|_{L_{6}}\big(\|z_2\|^2_{L_6}+\|x_1\|^2_{L_6}+\|x_2\|^2_{L_6}\big) +\|x_1-x_2\|^2_{1}\Big)
\end{aligned}
\end{equation*}
after also  dropping redundant terms in the left hand side.
As $x_1$, $x_2$ are in $\Hoz$ by assumption and $z_2$  is in $\Hoz$
by \eqref {eq:H2bnd},  it follows from \eqref{eq:sob} that
$T_y$ is continuous on $\Hoz$. 
\end{proof}
\begin{lem}\label{lem:zeta}
Assume that $4\dt \beta^4<1$ and
that $y\in \dom(A^{-3/2})$. If $\zeta\in [0,1]$ and $u=\zeta
T_y(u)$, then for some $M>0$ it must hold that $\|u\|_1\leq M\|A^{-3/2}y\|$.
\end{lem}
\begin{proof}
It is trivial for $\zeta=0$ so assume $0<\zeta\leq 1$ and write $\frac u\zeta=T(u)$ and substitute $z$ for $\frac u\zeta$ and
$x$ for $u$ in \eqref{eq:inner} and take $v=u$. %
Then,
\begin{equation*}
\frac 1\zeta \left(\|u\|^2_{-1}+k\|u\|_1^2+k\|u\|^4_{L_4}\right)=\langle
  y, u\rangle_{-1}+k\beta^2\|u\|^2.
\end{equation*}
After multiplication with $\zeta$ and similar arguments as above we get
\begin{equation*}
\begin{aligned}
&\|u\|^2_{-1}+k\|u\|_1^2\leq \frac
C\epsilon\|y\|^2_{-3}+\epsilon
\zeta^2\|u\|_1^2+\|u\|^2_{-1}+\frac{(\zeta k)^2 \beta^4}4\|u\|_1^2:=G(\zeta).%
\end{aligned}
\end{equation*}
It holds that $\sup_{0<\zeta\leq 1}G(\zeta)=G(1)$ so with $\zeta=1$ we see that under the assumption on $\dt$ we may pick
$0<\epsilon<\dt(1-\dt\beta^4/4)$ to achieve the desired result.
\end{proof}
\begin{proof}[Proof of Theorem \ref{thm:ell}]
Existence in $\Hoz$ follows immediately from Lemmata \ref{lem:comp},
-- \ref{lem:zeta} and uniqueness from \eqref{eq:finalbnd}. That the
solution is in $H^2$ is a result of \eqref{eq:H2bnd}. 
\end{proof}
\section{Extensions and future work}
The method above
generalises to e.g. homogeneous Neumann boundary conditions as in
\cite{Stig} and to arbitrary odd order polynomial $f$
with positive leading coefficient, cf. \cite{Prato},
but the  target non-linearity in the  Cahn-Hilliard context, the
logarithmic potential $f(s)=\log\left((1+s)/(1-s)\right)-\beta^2s$ remains a
challenge. %

Error analysis for the stochastic Cahn-Hilliard equation is
 performed in \cite{CHC}. A proof of strong convergence
inspired by \cite{KLLAC}, where the stochastic Allen-Cahn
(SAC) equation is treated, is given. %
To show the rate of convergence remains a challenge 
(see 
 \cite{KLLpre} for the SAC equation). So does fully discrete schemes. %

A drawback with the proof in this paper is that it does not
come with a constructive algorithm to find a solution. When Banach's
fixed point theorem is utilised fixed point iteration comes for free. With Schaefer's fixed point
theorem this is no longer the case and a numerical method must be
given and analysed.

\section*{Acknowledgements}
F. Lindgren was supported by JSPS KAKENHI Grant Number 15K45678.
\footnotesize
\setlength{\bibitemsep}{0pt}
\printbibliography 

\end{document}